\documentclass{amsart}
\usepackage[normalem]{ulem}%\uline,\uuline,\uwave,\sout,\xout
\usepackage[usenames]{color}
\usepackage[abbrev]{amsrefs}
\usepackage{amssymb}
\usepackage{tikz}
\usetikzlibrary{arrows.meta}
\usepackage{caption}
\usepackage{subcaption}

 \newcommand{\R}{\mathbb R}
 \newcommand{\Z}{\mathbb Z}

\DeclareMathOperator{\Int}{int}
\DeclareMathOperator{\cl}{cl}

\newcommand{\grid}{\mathcal{G}}
\newcommand{\diam}{\operatorname{diam}}
\newcommand{\fm}{{\mathcal F}}
\newcommand{\fmgg}[1]{\fm^{#1}}
\newcommand{\fmgk}[1]{\fmgg{\grid_{#1}}}
\newcommand{\fmg}{\fmgg\grid}
\newcommand{\hatfm}{{\widehat{\mathcal F}}}
\newcommand{\hatfmgg}[1]{\hatfm^{#1}}
\newcommand{\hatfmg}{\hatfmgg\grid}
\newcommand{\tildefm}{{\widetilde{\mathcal F}}}
\newcommand{\tildefmgg}[1]{\tildefm^{#1}}
\newcommand{\tildefmg}{\tildefmgg\grid}
\newcommand{\morsed}{{\mathcal M}}
\newcommand{\morsedg}[1]{\morsed^{#1}}
\newcommand{\morsedk}[1]{\morsedg{\grid_{#1}}}
\newcommand{\morses}{{\mathcal S}}
\newcommand{\morsesg}[1]{\morses^{#1}}
\newcommand{\morsesk}[1]{\morsesg{\grid_{#1}}}
\newcommand{\hm}{{\bar h}}
\newcommand{\morsea}{{\mathfrak M}}
\newcommand{\morseag}[1]{\morsea^{#1}}
\newcommand{\morseak}[1]{\morseag{\grid_{#1}}}
\newcommand{\ha}{{\mathfrak h}}
\newcommand{\morseg}{{\mathsf M}}
\newcommand{\morsegg}[1]{\morseg^{#1}}
\newcommand{\morsegk}[1]{\morsegg{\grid_{#1}}}

\newcommand{\hmg}{\mathsf h}

\newcommand{\incomp}{\not\gtreqqless}
\newcommand{\per}{\operatorname{per}}

\theoremstyle{plain} \newtheorem{thm}{Theorem}
\newtheorem{cor}[thm]{Corollary} \newtheorem{prop}[thm]{Proposition}
\newtheorem{lemma}[thm]{Lemma}

\theoremstyle{definition} \newtheorem{defn}[thm]{Definition}

\newtheorem{ex}[thm]{Example} 

\theoremstyle{remark} \newtheorem{remark}[thm]{Remark}

\numberwithin{thm}{section}

\begin{document}

\title[Persistence of Morse decompositions]{Persistence of Morse decompositions over grid resolution for maps and time series}

  \author{Jim Wiseman}    \address{Agnes Scott College \\ Decatur, GA 30030} \email{jwiseman@agnesscott.edu}

\thanks{This work was supported by a grant from the Simons Foundation (282398, JW)}
\keywords{Morse decomposition, persistent homology, time series}

\subjclass[2020]{55N31, 37B20, 37B35, 37B65, 37E25, 37M10}

\begin{abstract}
We can approximate a continuous map $f:X\to X$ of a compact metric space by discretizing the space into a grid.  Through either the map itself or a time series, $f$ induces a multivalued grid map $\fm$.  The dynamical properties of $\fm$ depend on the resolution of the grid, and we study the persistence of these properties as we change the resolution.  In particular, we look at the persistence of Morse decompositions, at both the global  (Morse graph) and local (individual Morse set) levels, using several notions of persistence -- graph structure, persistent homology, and mixing properties.
\end{abstract}

\maketitle

%Morse decomposition for maps - use algorithmic CR terminology here and below
%
%Grid, combinatorial maps, Morse decomp - SCCs (or strongly connected PATH components? or just restrict to recurrent set?)
%
%Given grid, standard map definitions (from map, or data, or time series - Rn defn Kokobu?)
%
%grid refinements - global or local, but monotone
%
%Gives persistent (define carefully - start with coarsest (all one square) - is there a finest, or keep going?) combinatorial maps, and so persistent Morse decomps
%
%Measure changes globally and locally
%
%GLOBALLY - Morse graph: collapse SCCs to vertices, minimal partial order edges (mention complete partial order graph for later).  As grid gets coarser, vertices can merge and can get new edges.  
%
%Persistence - number of SCCs/vertices - get merge tree (elder rule).
%
%Tree structure.  Recursively create total order that gives total order at each level.***  Look at sublevel set homology. Minimal and complete partial order graphs give same H0.***
%
%(Mrozek does homology overall, for possibly touching sets.  We don't need - can do each SCC individually with cubical homology.)
%
%LOCALLY
%
%look at individual SCC.  Get bigger (new cube)?  Persistent cubical homology.  Internal dynamics change?:  Periodic to primitive?  More generally, different loops - characteristic polynomial.***
%

\section{Introduction}

In order to understand the dynamics of a continuous map $f:X\to X$ of a compact metric space, we can create a finite discretization of the space, then use a computer to create a multivalued map $\fm$ on the discretization to approximate $f$.
The relationship between the dynamics of $f$ and those of $\fm$ is an active area of research (see, for example, \cites{Osi,OsiC,aacr,Akin,lp,paG,Hunt,Norton}), and as the computational technology improves, understanding this relationship becomes increasingly important.
This is true whether the map $\fm$ is constructed directly, from the action of $f$ on the elements of the discretization, or indirectly, from a sample of the dynamics (a time series).

We will discretize the space using a grid (defined in  Section~\ref{sect:grids}).  The behavior of the map $\fm$ depends on the resolution of the grid, perhaps especially in the case of maps reconstructed from time series.  This is discussed in \cite{bmmp}: too coarse a grid will give only a very rough approximation of $f$, while too fine a grid can isolate each individual data point.  As stated there, it is an interesting problem to understand this dependence of the dynamics on grid resolution in the spirit of persistent topology, as in \cites{ejm,bejm,djkklm}.

This paper studies persistence of Morse decompositions as a step in that direction.  A Morse decomposition (\cite{Conley}) is a collection of invariant sets of recurrent points (Morse sets), such that all other points in the space move from one Morse set to another.  This notion is very useful in computational dynamics (see, for example, \cites{aacr,akkmop,bmmp,BK,BKMW,DK,ghmkk,KKV,mik}).  The persistent homology of Morse sets has been studied in \cite{Szymczak} (zero-dimensional persistence for vector fields) and more generally in \cite{djkklm}.  In \cite{djkklm}, the authors use Alexandrov topology to study the persistent homology of the Morse sets (where the persistence can be over different kinds of parameters).  In this paper we take a different approach, looking at the Morse decomposition more broadly, including the connections between the different Morse sets (the Morse graph) and the dynamics on the individual Morse sets.

The paper is organized as follows.  In Section~\ref{sect:grids}, we discuss grids on $X$ and the multivalued grid maps induced by the map $f:X\to X$.  We define recurrence and Morse decompositions in Section~\ref{sect:MorseDefn}, and define and prove persistence for grid maps and Morse decompositions in Section~\ref{sect:persist}.  We discuss Morse graphs and persistence of the global dynamics in Section~\ref{sect:globalPers}, and local persistence of individual Morse sets in Section~\ref{sect:localPers}.  Finally, in Section~\ref{sect:timeSeries}, we apply our results to grid maps generated by time series.

\section{Grids and induced maps}\label{sect:grids}

Let $X$ be a  compact metric space, and $f:X\to X$ a continuous map.  In many cases $X$ will be a subset of $\R^n$. We discretize the space using a grid.  In defining grids on $X$ and the induced maps, we mostly follow the notation of \cite{aacr}; see \cite{Day} for an introduction.

\begin{defn}[\cite{mrGrid}]
A \emph{grid} $\grid$ on $X$ is a finite collection of nonempty compact subsets of $X$ such that
	\begin{enumerate}
	\item $X = \bigcup_{G\in\grid}G$
	\item $G = \cl(\Int(G))$ for all $G\in\grid$
	\item $G \cap \Int(H) = \emptyset$ for all $G\ne H \in\grid$.
	\end{enumerate}

\end{defn}

When $X$ is a subset of $\R^n$, we  often take $\grid$ to be a cubical grid.  In general, we define the \emph{diameter}, or \emph{resolution}, of $\grid$ by $\diam(\grid) = \max_{G\in\grid} \diam(G)$.
The \emph{geometric realization $|\cdot|$} is a map from the power set of $\grid$ to the power set of $X$, given by $|\mathcal A|= \bigcup_{G\in\mathcal A}G$.

Let $\grid$ and $\grid'$ be two grids on $X$.  We say that \emph{$\grid$ refines $\grid'$}, or \emph{$\grid$ is a refinement of $\grid'$}, or \emph{$\grid'$ is a coarsening of $\grid$}, and write $\grid < \grid'$, if for every element $G\in\grid$ there is a $G'\in\grid'$ such that $G\subset G'$.

We can use the map $f$ directly to generate a multivalued map on $\grid$.  In section~\ref{sect:timeSeries}, we use time series to generate the map.

Given a compact metric space $X$, a continuous map $f: X\to X$, and a grid $\grid$ on $X$, we define the \emph{minimal multivalued map $\fmg$ associated to $f$ on} $\grid$ by $\fmg(G) := \{H \in\grid : H\cap f(G)\ne\emptyset\}$.  Note that $\fmg: \grid \rightrightarrows \grid$ is an outer approximation of $f:X \to X$, meaning that $f(G) \subset \Int(|\fmg(G|)$ for every $G\in\grid$ (\cite{aacr}*{Prop.~2.5}).

An \emph{orbit} for  a multivalued map $\fm:\grid\rightrightarrows\grid$ is a (possibly infinite or bi-infinite) sequence $\{G_i\}$ of grid elements such that $ G_{i+1} \in \fm(G_i)$ for each $i$.

We can identify a multivalued map $\fm:\grid\rightrightarrows\grid$ with a directed graph, also denoted $\fm$, with  the grid elements $G$ of $\grid$ as the vertices and an edge from $G$ to $H$ if $H\in\fm(G)$.  An orbit for the map corresponds to a path in the graph.

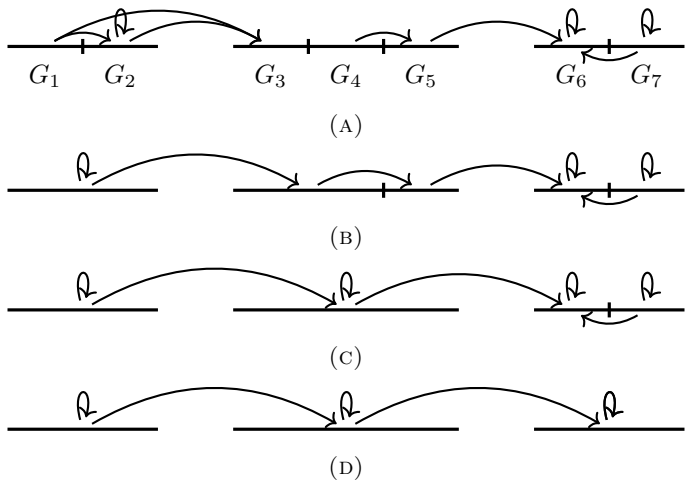
\begin{figure}

\begin{subfigure}{\textwidth}
\centering
\begin{tikzpicture}%[xscale=5]

\draw (0,0) edge[very thick] (2,0);
\draw[-][very thick] (1,-.1) -- (1,0.1);
\node [label=below:{$G_1$}] (A) at (.5,0) {};
\node  [label=below:{$G_2$}] (B)  at (1.5,0) {};

\draw (3,0) edge[very thick] (6,0);
\draw[-][very thick] (4,-.1) -- (4,0.1);
\draw[-][very thick] (5,-.1) -- (5,0.1);
\node  [label=below:{$G_3$}] (C) at (3.5,0) {};
\node  [label=below:{$G_4$}] (D)  at (4.5,0) {};
\node  [label=below:{$G_5$}] (E)  at (5.5,0) {};

\draw (7,0) edge[very thick] (9,0);
\draw[-][very thick] (8,-.1) -- (8,0.1);
\node  [label=below:{$G_6$}] (F)  at (7.5,0) {};
\node  [label=below:{$G_7$}] (G) at (8.5,0) {};

\draw [->] (B) edge[thick, loop above] (B);
\draw [->] (A) edge[thick, bend left] (B);
\draw [->] (A) edge[thick,bend left] (C);
\draw [->] (B) edge[thick,bend left] (C);
\draw [->] (D) edge[thick,bend left] (E);
\draw [->] (E) edge[thick,bend left] (F);
\draw (F) edge[thick, loop above] (F);
\draw (G) edge[thick, loop above] (G);
\draw [->] (G) edge[thick,bend left] (F);

\end{tikzpicture}
\caption{}
\label{fig:ex1}
\end{subfigure}

\begin{subfigure}{\textwidth}
\centering
\begin{tikzpicture}%[xscale=5]

\draw (0,0) edge[very thick] (2,0);
%\draw[-][very thick] (1,-.1) -- (1,0.1);
\node (A) at (1,0) {};
%\node  (B)  at (1.5,0) {};

\draw (3,0) edge[very thick] (6,0);
%\draw[-][very thick] (4,-.1) -- (4,0.1);
\draw[-][very thick] (5,-.1) -- (5,0.1);
\node   (C) at (4,0) {};
%\node  (D)  at (4.5,0) {};
\node   (E)  at (5.5,0) {};

\draw (7,0) edge[very thick] (9,0);
\draw[-][very thick] (8,-.1) -- (8,0.1);
\node   (F)  at (7.5,0) {};
\node  (G) at (8.5,0) {};

\draw [->] (A) edge[thick, loop above] (A);
%\draw [->] (A) edge[thick, bend left] (B);
\draw [->] (A) edge[thick,bend left] (C);
%\draw [->] (B) edge[thick,bend left] (C);
\draw [->] (C) edge[thick,bend left] (E);
\draw [->] (E) edge[thick,bend left] (F);
\draw (F) edge[thick, loop above] (F);
\draw (G) edge[thick, loop above] (G);
\draw [->] (G) edge[thick,bend left] (F);

\end{tikzpicture}
\caption{}
\label{fig:ex2}
\end{subfigure}

\begin{subfigure}{\textwidth}
\centering
\begin{tikzpicture}%[xscale=5]

\draw (0,0) edge[very thick] (2,0);
%\draw[-][very thick] (1,-.1) -- (1,0.1);
\node (A) at (1,0) {};
%\node  (B)  at (1.5,0) {};

\draw (3,0) edge[very thick] (6,0);
%\draw[-][very thick] (4,-.1) -- (4,0.1);
%\draw[-][very thick] (5,-.1) -- (5,0.1);
\node   (C) at (4.5,0) {};
%\node  (D)  at (4.5,0) {};
%\node   (E)  at (5.5,0) {};

\draw (7,0) edge[very thick] (9,0);
\draw[-][very thick] (8,-.1) -- (8,0.1);
\node   (F)  at (7.5,0) {};
\node  (G) at (8.5,0) {};

\draw [->] (A) edge[thick, loop above] (A);
%\draw [->] (A) edge[thick, bend left] (B);
\draw [->] (A) edge[thick,bend left] (C);
%\draw [->] (B) edge[thick,bend left] (C);
\draw [->] (C) edge[thick,loop above] (C);
\draw [->] (C) edge[thick,bend left] (F);
\draw (F) edge[thick, loop above] (F);
\draw (G) edge[thick, loop above] (G);
\draw [->] (G) edge[thick,bend left] (F);

\end{tikzpicture}
\caption{}
\label{fig:ex3}
\end{subfigure}

\begin{subfigure}{\textwidth}
\centering
\begin{tikzpicture}%[xscale=5]

\draw (0,0) edge[very thick] (2,0);
%\draw[-][very thick] (1,-.1) -- (1,0.1);
\node (A) at (1,0) {};
%\node  (B)  at (1.5,0) {};

\draw (3,0) edge[very thick] (6,0);
%\draw[-][very thick] (4,-.1) -- (4,0.1);
%\draw[-][very thick] (5,-.1) -- (5,0.1);
\node   (C) at (4.5,0) {};
%\node  (D)  at (4.5,0) {};
%\node   (E)  at (5.5,0) {};

\draw (7,0) edge[very thick] (9,0);
%\draw[-][very thick] (8,-.1) -- (8,0.1);
\node   (F)  at (8,0) {};
%\node  (G) at (8.5,0) {};

\draw [->] (A) edge[thick, loop above] (A);
%\draw [->] (A) edge[thick, bend left] (B);
\draw [->] (A) edge[thick,bend left] (C);
%\draw [->] (B) edge[thick,bend left] (C);
\draw [->] (C) edge[thick,loop above] (C);
\draw [->] (C) edge[thick,bend left] (F);
\draw (F) edge[thick, loop above] (F);
\draw (F) edge[thick, loop above] (F);
%\draw [->] (G) edge[thick,bend left] (F);

\end{tikzpicture}
\caption{}
\label{fig:ex4}
\end{subfigure}

\caption{Grid maps.}
\label{fig:gridMapEx}
\end{figure}

\begin{ex}
In Figure~\ref{fig:gridMapEx}, we have a simple one-dimensional example, showing only the grid elements and edges relevant to the example.  The grids grow increasingly coarser.  Going from (A) to (B), for example, $G_1$ and $G_2$ are merged, as are $G_3$ and $G_4$.

\end{ex}

We define the \emph{inverse} of $\fm$ by $\fm^{-1}(G):=\{H \in \grid : G\in \fm(H)\}$.  We say that a subset $\morses \subset \grid$ is \emph{invariant} if $\morses \subset \fm(\morses)$ and $\morses \subset \fm(^{-1}\morses)$.

We say that $\fm$ is \emph{closed} if $\fm(G)\ne\emptyset$ and $\fm^{-1}(G)\ne\emptyset$  for every $G\in\grid$.  Equivalently, each vertex has at least one edge coming in and one edge going out, that is, there are no stranded vertices.  If $f$ is surjective (in particular, if $f$ is a homeomorphism), then the minimal multivalued map associated to $f$, $\fmg$, is closed (\cite{aacr}*{Prop.~3.2}).  More generally, since there are no bi-infinite orbits through a stranded vertex, we can remove all stranded vertices from $\fm$ without affecting the dynamics.  Thus we can assume that each $\fm$ is closed.

\section{Morse decompositions}\label{sect:MorseDefn}

There are many notions of recurrence for dynamical systems (\cites{Akin,Alongi}), but recurrence is simpler for maps $\fm:\grid\rightrightarrows\grid$.  We say that a grid element $G$ is \emph{recurrent} if there is a nontrivial orbit from $G$ to itself.  Grid elements $G$ and $H$ are equivalent if there are orbits from $G$ to $H$ and from $H$ to $G$.  This gives an equivalence relation on the recurrent set.  The equivalence classes correspond exactly to the nontrivial strongly connected components of the graph $\fm$.

\begin{defn}[\cite{mrGrid}]
A \emph{Morse decomposition} $\morsed$ for a closed multivalued map $\fm:\grid\rightrightarrows\grid$ is a collection of invariant sets $\morses_1,\ldots,\morses_n$, called Morse sets, with a partial order $\succeq$ such that for any bi-infinite orbit $\{G_i\}_{i\in\Z}$, either there exists an $\morses_j$ such that $G_i\in\morses_j$ for all $i$, or there exist integers $i^- < i^+$ and Morse sets $\morses_j \succ \morses_k$ such that $G_i \in \morses_j$ for all $i\leq i^-$ and $G_i \in \morses_k$ for all $i\geq i^+$.
That is, every complete orbit either is contained in one of the Morse sets, or begins in one Morse set and ends in another, lower Morse set.

The equivalence classes of the recurrent set (the strongly connected components of $\fm$) are the Morse sets of the finest Morse decomposition $\morsedg{\grid}$ (\cite{djkklm}*{Thm.~4.1}).  
The partial order $\succeq$ is  given by $\morses_j \succeq \morses_k$ if there is an orbit going from $\morses_j$ to $\morses_k$.  We observe that in this case, if there is an orbit from $\morses_j$ to $\morses_k$ and an orbit from $\morses_k$ to $\morses_l$, then there is an orbit from $\morses_j$ to $\morses_l$, because there are orbits between any two points within a Morse set.  This differs from the case of Morse decompositions for maps, where there can be an orbit from one Morse set to another, and from the second to a third, but not from the first to the third.
%The partial order $\succeq$ is the transitive extension of the relation $\geq$ given by $\morses_j \geq \morses_k$ if there is an orbit going from $\morses_j$ to $\morses_k$.

\begin{ex}\label{ex:MorseSets}
In Figure~\ref{fig:gridMapEx}(A), the Morse sets are $\{G_2\}$, $\{G_6\}$, and $\{G_7\}$.  In (B), the first has grown to the merged grid element $\{G_1\cup G_2\}$.  In (C), we add the Morse set consisting of the merged $G_3\cup G_4 \cup G_5$.  Finally, in (D) the Morse sets $\{G_6\}$ and $\{G_7\}$ are merged into the single Morse set $\{G_6 \cup G_7\}$.
\end{ex}

We are interested not in  just the Morse sets themselves, but in the dynamics of $\fm$ on each set as well.  On each $\morses_j$, $\fm$ induces the restriction map $\fm_j:\morses_j\rightrightarrows\morses_j$, given by $\fm_j(G)=\fm(G)\cap\morses_j$.  Equivalently, $\fm_j$ is the induced subgraph of $\fm$ on the vertices in $\morses_j$.
We define $\morsea$, the \emph{augmented Morse decomposition}, to be the Morse decomposition along with the induced maps, $$\morsea = (\morsed, \{\fm_j:\morses_j\rightrightarrows\morses_j\}) = (\{\morses_j\},\succeq, \{\fm_j:\morses_j\rightrightarrows\morses_j\}).$$

%We denote by $(\morsed,\fm)$ the set of pairs $(\morses_j,\fm_j:\morses_j\rightrightarrows\morses_j)$, or, equivalently, the set of induced subgraphs $\fm_j$ of $\fm$.

\end{defn}

Topological dynamics on graphs is discussed more generally in \cite{AW}.  
It is shown in \cites{aacr,OsiC,Akin}, in slightly varying contexts, that as the diameter of the grid goes to 0, the recurrent set of $\fmg$ limits on the chain recurrent set for the underlying map $f$.
For more on the role of Morse decompositions in computational dynamics, see, for example, \cites{BK,ghmkk,KKV,DK,mik}.

\section{Persistence}  \label{sect:persist}

Let $\{\grid_k\}$ be a collection of grids for $X$, partially ordered by refinement.  Our goal is to show that this collection gives persistence of augmented Morse decompositions, in the following sense, which we will make precise.

For a given grid $\grid_k$, with associated minimal multivalued map $\fmgk{k}$, denote by $$\morseak{k} = (\morsedk{k}, \{\fmgk{k}_j: \morsesk{k}_j \rightrightarrows  \morsesk{k}_j\}) = (\{\morsesk{k}_j\}, \succeq^{\grid_k}, \{\fmgk{k}_j: \morsesk{k}_j \rightrightarrows  \morsesk{k}_j\})  $$ the corresponding finest augmented 
Morse decomposition, arising from the strongly connected components of $\fmgk{k}$.  Then for every $\grid_{k_1} < \grid_{k_2}$, there is an appropriately defined morphism $\ha_{\grid_{k_1}}^{\grid_{k_2}}: \morseak{k_1} \to \morseak{k_2}$ that respects the Morse sets, the partial order, and the induced maps.  If $\grid_{k_1} < \grid_{k_2} < \grid_{k_3}$, then $\ha_{\grid_{k_2}}^{\grid_{k_3}} \circ \ha_{\grid_{k_1}}^{\grid_{k_2}} = \ha_{\grid_{k_1}}^{\grid_{k_3}}$.

We will proceed as follows.  It is straightforward to show that grid refinement gives persistence of the minimal multivalued maps associated to the grids.  It then follows that the multivalued map persistence induces persistence of the augmented Morse decompositions.  (There are somewhat analogous results, using covers instead of grids, in \cite{Osi}.)

\begin{defn}
Let $\grid$ and $\grid'$ be grids, and $\fm:\grid\rightrightarrows\grid$ and $\fm':\grid'\rightrightarrows\grid'$ be multivalued maps.  A (single-valued) map $h:\grid\to\grid'$ is a \emph{grid map morphism} if $h(\fm(G)) \subset \fm'(h(G))$ for all $G\in\grid$.  Equivalently, $h$ is a directed graph homomorphism from $\fm$ to $\fm'$.
\end{defn}

\begin{prop}\label{prop:gridmappersist}
Let $f:X\to X$ be a continuous map of a compact metric space, and let $\{\grid_k\}$ be a collection of grids for $X$, partially ordered by refinement.  For each $k$, let $\fmgk{k}:\grid_k\rightrightarrows \grid_k$ be the minimal multivalued map associated to $f$ on $\grid_k$.
Then for every $\grid_{k_1} < \grid_{k_2}$, there is a morphism $h_{k1}^{k2}:\grid_{k_1} \to \grid_{k_2}$, and if $\grid_{k_1} < \grid_{k_2} < \grid_{k_3}$, then $h_{k_2}^{k_3}\circ h_{k_1}^{k_2} = h_{k_1}^{k_3}$.
\end{prop}

\begin{proof}
Take grids $\grid < \grid'$.  Since $\grid$ is a refinement of $\grid'$, by definition for any $G\in\grid$ there is a $G'\in\grid'$ such that $G\subset G'$; define $h$ by $h(G) := G'$.  If $H\in\fm(G)$, then $H\cap f(G) \ne \emptyset$.  Since $G\subset h(G)$ and $H \subset h(H)$, we have $h(H) \cap f(h(G)) \ne \emptyset$, so $h(H)\in\fm'(h(G))$.  Since $H$ was an arbitrary element of $\fm(G)$, we have $h(\fm(G))\subset\fm'(h(G))$.  That $h_{k_2}^{k_3}\circ h_{k_1}^{k_2} = h_{k_1}^{k_3}$ follows from the definition; the composition of inclusions is inclusion.
\end{proof}

\begin{ex}
In Figure~\ref{fig:gridMapEx}, we see that at as the grid goes coarser, the inclusion of a grid element into a larger, merged grid element induces persistence of grid maps.
\end{ex}

\begin{defn}
Let $\grid$ and $\grid'$ be grids, and $\fm:\grid\rightrightarrows\grid$ and $\fm':\grid'\rightrightarrows\grid'$  multivalued maps.  Let $\morsedg\grid = (\{\morsesg{\grid}_j\},\succeq^{\grid})$ and $\morsedg{\grid'}= (\{\morsesg{\grid'}_j\},\succeq^{\grid'})$ be the corresponding finest Morse decompositions, arising from the strongly connected components of $\fm$ and $\fm'$.  A map $\hm:\{\morsesg{\grid}_j\} \to \{\morsesg{\grid'}_j\}$ is a \emph{Morse decomposition morphism} if $\hm(\morsesg{\grid}_j) \succeq^{\grid'} \hm(\morsesg{\grid}_k)$ whenever $\morsesg{\grid}_j \succeq^{\grid} \morsesg{\grid}_k$.  Let $\morseag{\grid} = (\morsedg{\grid}, \{\fmgg{\grid}_j:\morsesg{\grid}_j \rightrightarrows \morsesg{\grid}_j\})$ and $\morseag{\grid'} = (\morsedg{\grid'}, \{\fmgg{\grid'}_j:\morsesg{\grid'}_j \rightrightarrows \morsesg{\grid'}_j\})$ be the corresponding augmented Morse decompositions.  An \emph{augmented Morse decomposition morphism} $\ha:\morseag{\grid}\to \morseag{\grid'}$ is a pair $(\hm,\{h_j\})$, where $\hm$ is a morphism between $\morsedg{\grid}$ and $\morsedg{\grid'}$, and for each $\morsesg{\grid}_j$, $h_j: \morsesg{\grid}_j\to\hm(\morsesg{\grid}_j)$ is a grid map morphism.
\end{defn}

\begin{lemma}\label{lem:augMorsePersist}
Let $\{\grid_k\}$ be a collection of grids, partially ordered by refinement, with maps $\fmgk{k}:\grid_k\rightrightarrows \grid_k$.  Then persistence for grid maps induces persistence for the corresponding finest augmented Morse decompositions.  That is, if for every $\grid_{k_1} < \grid_{k_2}$, there is a morphism $h_{k1}^{k2}:\grid_{k_1} \to \grid_{k_2}$ such that  $h_{k_2}^{k_3}\circ h_{k_1}^{k_2} = h_{k_1}^{k_3}$ for any $\grid_{k_1} < \grid_{k_2} < \grid_{k_3}$, then for every $\grid_{k_1} < \grid_{k_2}$, there is a morphism $\ha_{\grid_{k_1}}^{\grid_{k_2}}: \morseak{k_1} \to \morseak{k_2}$ such that $\ha_{\grid_{k_2}}^{\grid_{k_3}} \circ \ha_{\grid_{k_1}}^{\grid_{k_2}} = \ha_{\grid_{k_1}}^{\grid_{k_3}}$ for any $\grid_{k_1} < \grid_{k_2} < \grid_{k_3}$.
\end{lemma}

\begin{proof}
This is essentially the fact that directed graph homomorphisms preserve strong connectedness.  Take grids $\grid < \grid'$ and maps $\fm:\grid\rightrightarrows\grid$ and $\fm':\grid'\rightrightarrows\grid'$, with morphism $h:\grid\to\grid'$.  Then  any $\fm$-orbit $(G_0=G,G_1,\ldots,G_n=H)$ from $G$ to $H \in \grid$ gives an $\fm'$-orbit $(h(G_0)=h(G),h(G_1),\ldots,h(G_n)=h(H))$ from $h(G)$ to $h(H) \in \grid'$.  Thus if $G$ is recurrent for $\fm$, then $h(G)$ is recurrent for $\fm'$, and
 if $G$ and $H$ are in the same Morse set $\morsesg{\grid}_j$, then $h(G)$ and $h(H)$ are in the same Morse set $\morsesg{\grid'}_{j'}$, 
and we can define $\hm$ by $\hm(\morsesg{\grid}_j) := \morsesg{\grid'}_{j'}$, where $h(G) \in \morsesg{\grid'}_k$ for all $G\in\morsesg{\grid}_j$.  Furthermore, if there is an $\fm$-orbit from 
$\morsesg{\grid}_j $ to $\morsesg{\grid}_k$, then there is an $\fm'$-orbit from $\hm(\morsesg{\grid}_j)$ to $\hm(\morsesg{\grid}_k)$, and so $\hm(\morsesg{\grid}_j) \succeq^{\grid'} \hm(\morsesg{\grid}_k)$ 
whenever $\morsesg{\grid}_j \succeq^{\grid} \morsesg{\grid}_k$.  Thus $\hm$ is a Morse decomposition morphism.  

For each $\morsesg{\grid}_j$, we can now define $h_j: \morsesg{\grid}_j\to\hm(\morsesg{\grid}_j)$ as the restriction of $h$ to $\morsesg{\grid}_j$.

The identity $\ha_{\grid_{k_2}}^{\grid_{k_3}} \circ \ha_{\grid_{k_1}}^{\grid_{k_2}} = \ha_{\grid_{k_1}}^{\grid_{k_3}}$ for any $\grid_{k_1} < \grid_{k_2} < \grid_{k_3}$ follows from the definition.
\end{proof}

As an immediate corollary, we have persistence over grid refinement for the finest Morse decompositions.

\begin{thm}
Let $f:X\to X$ be a continuous map of a compact metric space, and let $\{\grid_k\}$ be a collection of grids for $X$, partially ordered by refinement.  For each $k$, let $\fmgk{k}:\grid_k\rightrightarrows \grid_k$ be the minimal multivalued map associated to $f$ on $\grid_k$.  Then for every $\grid_{k_1} < \grid_{k_2}$, there is a morphism $\ha_{\grid_{k_1}}^{\grid_{k_2}}: \morseak{k_1} \to \morseak{k_2}$ such that $\ha_{\grid_{k_2}}^{\grid_{k_3}} \circ \ha_{\grid_{k_1}}^{\grid_{k_2}} = \ha_{\grid_{k_1}}^{\grid_{k_3}}$ for any $\grid_{k_1} < \grid_{k_2} < \grid_{k_3}$.
\end{thm}

\begin{ex}
In Figure~\ref{fig:gridMapEx} and Example~\ref{ex:MorseSets}, we see that at as the grid goes coarser, the inclusion of a grid element into a larger, merged grid element induces persistence of augmented Morse decompositions.
\end{ex}

\begin{remark}
It may seem more natural to have morphisms going the other way, from the objects corresponding to the coarser grid to those corresponding to the finer grid, in particular so that we can take finer and finer grids.  One obvious way to achieve this is to take cohomology instead when we look at homology in the following sections.  

More generally, 
there are essentially two reasons that we do not get grid map morphisms $\grid'\to\grid$ for $\grid < \grid'$.  First, a $G'\in\grid'$ can correspond to multiple elements of $\grid$, and second, an edge in $\grid'$ may not correspond to an edge in $\grid$.  We can address both of these issues by looking at a kind of dual graph to $\fmg$, $(\fmg)^*$.  The vertices of  $(\fmg)^*$ are the subsets of the vertices of $\fmg$, and there is an edge from $\{G_1,\ldots,G_m\}$ to $\{H_1,\ldots,H_n\}$ if for all $i$ and $j$ there is no edge $G_i\to H_j$ in $\fmg$.  Then there is a directed graph homomorphism  $ (\fmgg{\grid'})^* \to (\fmg)^*$ for $\grid < \grid'$, which gives persistence, as discussed in this section.

A disadvantage of this approach is that the vertex set is very large, but we can take advantages of symmetries to greatly reduce the computational complexity.
The details of these dual graphs and the relation of their properties, such as Morse decompositions, to the properties of $\fmg$ will be the subject of future work.
\end{remark}

\section{Global persistence}\label{sect:globalPers}

\subsection{Morse graphs}

For a given grid $\grid$, we can define the associated Morse graph $\morsegg\grid$ by starting with the graph $\fmg$ and collapsing each Morse set $\morsesg\grid_i$ to a single vertex.  (This is essentially the condensation of $\fmg$, or the quotient graph by the equivalence relation on the recurrent set.)  This gives a directed acyclic graph that describes the gradient-like (non-recurrent) behavior of the system (see \cite{akkmop}).  More precisely, we have the following definition.

\begin{defn}
The Morse graph $\morsegg\grid$ for the grid $\grid$ has the vertex set $\{v_i = \morsesg\grid_i\}$, with an edge from $v_i$ to $v_j$ if $\morsesg\grid_i\succ\morsesg\grid_j$, that is, if there is an orbit from $\morsesg\grid_i$ to $\morsesg\grid_j$.
\end{defn}

The vertices inherit the partial order $\succeq$ on the Morse sets.  We write $v \incomp v'$ if the two vertices are incomparable, that is, $v \nsucceq v'$ and $v' \nsucceq v$.  Observe that a Morse graph has no loops and no multiple edges.  In addition, the edges are transitive:  if there are edges $v_i\to v_j$ and $v_j\to v_k$, then there is an edge $v_i\to v_k$.

  It is clear that we have persistence over grid refinement of Morse graphs, since grid map morphisms preserve orbits.  To be precise, we have the following proposition.

\begin{prop}\label{prop:MorseGraphPersist}
Let $f:X\to X$ be a continuous map of a compact metric space, and let $\{\grid_k\}$ be a collection of grids for $X$, partially ordered by refinement.  For each $k$, let $\fmgk{k}:\grid_k\rightrightarrows \grid_k$ be the minimal multivalued map associated to $f$ on $\grid_k$.  Then for every $\grid_{k_1} < \grid_{k_2}$, there is a directed graph homomorphism $\hmg_{\grid_{k_1}}^{\grid_{k_2}}: \morsegk{k_1} \to \morsegk{k_2}$ such that $\hmg_{\grid_{k_2}}^{\grid_{k_3}} \circ \hmg_{\grid_{k_1}}^{\grid_{k_2}} = \hmg_{\grid_{k_1}}^{\grid_{k_3}}$ for any $\grid_{k_1} < \grid_{k_2} < \grid_{k_3}$.
\end{prop}

\begin{figure}

\begin{subfigure}{\textwidth}
\centering
\begin{tikzpicture}%[xscale=5]

\node[draw,circle] (v1) at (0,0) {$v_1$};
\node[draw,circle] (v2) at (2,0) {$v_2$};
\node[draw,circle] (v3) at (4,0) {$v_3$};
\draw [-Latex, thick] (v3) -- (v2);

\end{tikzpicture}
\caption{}
\label{fig:exM1}
\end{subfigure}

\begin{subfigure}{\textwidth}
\centering
\begin{tikzpicture}%[xscale=5]

\node[draw,circle] (v1) at (0,0) {$v_1'$};
\node[draw,circle] (v2) at (2,0) {$v_2$};
\node[draw,circle] (v3) at (4,0) {$v_3$};
\draw [-Latex, thick] (v3) -- (v2);
\draw [-Latex, thick] (v1) -- (v2);

\end{tikzpicture}
\caption{}
\label{fig:exM2}
\end{subfigure}

\begin{subfigure}{\textwidth}
\centering
\begin{tikzpicture}%[xscale=5]

\node[draw,circle] (v1) at (0,0) {$v_1'$};
\node[draw,circle] (v4) at (2,0) {$v_4$};
\node[draw,circle] (v2) at (4,0) {$v_2$};
\node[draw,circle] (v3) at (6,0) {$v_3$};
\draw [-Latex, thick] (v3) -- (v2);
\draw [-Latex, thick] (v1) -- (v4);
\draw [-Latex, thick] (v4) -- (v2);
\draw (v1) edge[-Latex, thick, bend left]  (v2);

\end{tikzpicture}
\caption{}
\label{fig:exM3}
\end{subfigure}

\begin{subfigure}{\textwidth}
\centering
\begin{tikzpicture}%[xscale=5]

\node[draw,circle] (v1) at (0,0) {$v_1'$};
\node[draw,circle] (v4) at (2,0) {$v_4$};
\node[draw,circle] (v2) at (4,0) {$v_2'$};
%\node[draw,circle] (v3) at (6,0) {$v_3$};
%\draw [-Latex, thick] (v3) -- (v2);
\draw [-Latex, thick] (v1) -- (v4);
\draw [-Latex, thick] (v4) -- (v2);
\draw (v1) edge[-Latex, thick, bend left]  (v2);

\end{tikzpicture}
\caption{}
\label{fig:exM4}
\end{subfigure}

\caption{Persistence of Morse graphs}
\label{fig:MorseGraph}
\end{figure}
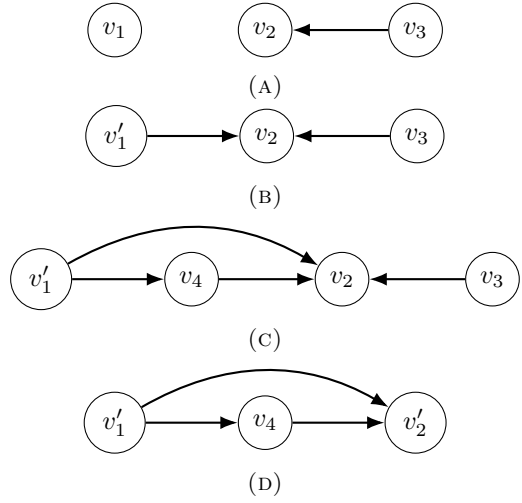

\begin{ex}\label{ex:MorseGraphs}
In  Figure~\ref{fig:MorseGraph}, we see the Morse graphs from the example in Figure~\ref{fig:gridMapEx} and Example~\ref{ex:MorseSets}.
\end{ex}

%In fact, a slightly simpler graph gives a little more information about the dynamics.
%
%\begin{defn}
%The reduced Morse graph $\rmorsegg\grid$ for the grid $\grid$ has the vertex set
% $\{v_i = \morsesg\grid_i\}$, with an edge from $v_i$ to $v_j$  if there is an orbit from $\morsesg\grid_i$ directly to $\morsesg\grid_j$, that is, an orbit from $\morsesg\grid_i$ directly to $\morsesg\grid_j$ that does not pass through any other Morse set $\morsesg\grid_k$.
%\end{defn}
%
%The Morse graph can be recovered from the reduced Morse graph, but not vice versa.  We do not have persistence for reduced Morse graphs, as the following example shows.  However, as we will see, we can use the simpler reduced Morse graph to compute persistent objects.
%
%\begin{ex}
%Ex57
%\end{ex}

In \cites{akkmop,BGHKMOP,AGGKMO} the authors discuss Morse-Conley graphs, stronger versions of Morse graphs containing information about a range of parameters as well as about the dynamics on individual Morse sets, in the form of the Conley index, and they consider local refinements of the grid in order to obtain more detailed information.

\subsection{Merge trees for Morse graph vertices}

(A general reference for merge trees and persistent homology is \cite{EH}.)
Example~\ref{ex:MorseGraphs} shows that passing to a coarser grid can affect the Morse graph by adding an edge, adding a vertex, or merging  vertices.  (When vertices $\morses_i$ and $\morses_j$ are merged, that is, when $\hm(\morses_i)=\hm(\morses_j)$,  the affected edges are also merged; any edges from the set $\hm^{-1}(\hm(\morses_i))$ to a given vertex are merged into one, as are any edges from a given vertex to $\hm^{-1}(\hm(\morses_i))$.)
It is easy to see that these are the only possible changes.

\begin{prop} \label{prop:MorseGraphChanges}
Let $\grid < \grid'$ be two grids.  Then $\morsegg{\grid'}$ can be obtained from $\morsegg{\grid}$ by a series of the following operations:
\begin{enumerate}
\item Adding a vertex.
\item Adding an edge.
\item Merging vertices.
\end{enumerate}
\end{prop}

\begin{proof}
As we see in  Figure~\ref{fig:gridMapEx}  and Example~\ref{ex:MorseGraphs}, passing from $\fmg$ to $\fmgg{\grid'}$ can create new orbits, and thus add a vertex by making a set recurrent, add an edge by creating an orbit from one Morse set to another, or merge vertices by creating orbits between Morse sets or merging the grid elements in them.
The fact that we cannot lose vertices or edges (without merging) follows from Proposition~\ref{prop:MorseGraphPersist}; every vertex is mapped to a vertex and every edge to an edge.
\end{proof}

\begin{figure}

\begin{tikzpicture}[
dot/.style = {circle, fill, minimum size=5pt,
              inner sep=0pt, outer sep=0pt}
              ]

\node (gd) at (0,0) {$\grid_{\tt D}$};
\node[dot] (20) at (2,0) {};
\node[dot] (30) at (3,0) {};
\node[dot] (50) at (5,0) {};

\node (gc) at (0,1) {$\grid_{\tt C}$};
\node[dot] (21) at (2,1) {};
\node[dot] (31) at (3,1) {};
\node[dot] (41) at (4,1) {};
\node[dot] (61) at (6,1) {};

\node (gb) at (0,2) {$\grid_{\tt B}$};
\node[dot] (22) at (2,2) {};
\node[dot] (42) at (4,2) {};
\node[dot] (62) at (6,2) {};

\node (ga) at (0,3) {$\grid_{\tt A}$};
\node[dot] (23) at (2,3) {};
\node[dot] (43) at (4,3) {};
\node[dot] (63) at (6,3) {};

\draw (20) edge[thick] (23);
\draw (30) edge[thick] (31);
\draw (50) edge[thick] (41);
\draw (50) edge[thick] (61);
\draw (41) edge[thick] (43);
\draw (61) edge[thick] (63);

\end{tikzpicture}

\caption{Merge tree for Morse graphs}
\label{fig:mergeTree}
\end{figure}
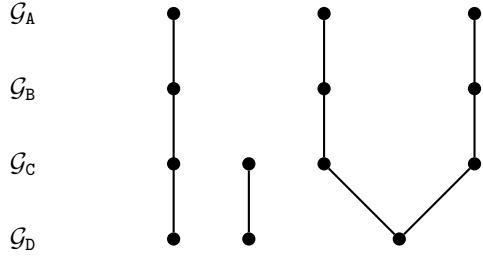

The simplest way of keeping track of the changes in the Morse graphs is by looking only at the vertices, and ignoring the edges.  For an ordered set of grids $\{\grid_k\}_{k=1}^N$, we create a merge tree as follows.  For each $\grid_k$, we draw the vertices of $\morsegk{k_1}$ as points.  We draw an edge from the point $\morses_i^{\grid_k}$ in level $k$ to the point $\morses_j^{\grid_{k+1}}$ in level $k+1$ if $\hm(\morses_i^{\grid_k}) = \morses_j^{\grid_{k+1}}$.
Figure~\ref{fig:mergeTree} shows the merge tree for the Morse graphs in Example~\ref{ex:MorseGraphs}.

\subsection{Persistent homology}

The merge tree is an easy way to measure the changes in the Morse graphs, but it loses information about the gradient-like structure since it does not take into account the edges of the graphs.  A simple way to do that is to count the cycles in the graph.  (Note that while a Morse graph is acyclic as a directed graph, it may have cycles as an undirected graph.)  

Cycles are concatenations of forward and backward orbit segments.  The most dynamically significant cycles are created when there are two (directed) paths of the same length from a vertex $v_i$ to a vertex $v_j$.  We can count these by looking at the (directed) adjacency matrix $A$ for the Morse graph $\morseg$.  (Since $\morseg$ is acyclic, $A$ will be nilpotent.)  The $i,j$ entry of $A^m$ gives the number of (directed) paths from $v_i$ to $v_j$, so if it is greater than 1, then there is a cycle:  follow one path from $v_i$ to $v_j$, then the other backwards from $v_j$ back to $v_i$.  Unfortunately, this method can overcount the number of such cycles.  If we have two paths of length $m_1$ from $v_i$ to $v_j$ and two of length $m_2$ from $v_j$ to $v_k$, then $A^{m_1+m_2}$ will detect four paths from $v_i$ to $v_k$, implying multiple cycles, when in fact there are just combinations of the two existing cycles.  We can address this overcounting algebraically, in homology.

Homology for graphs is particularly simple:  $H_0$ counts the connected components, and $H_1$ counts the cycles algebraically.  (See \cite{Massey}*{\S VII.3}.) For simplicity, we will use $\R$ coefficients.    Persistence for Morse graphs induces persistence for homology:

\begin{prop}
Let $f:X\to X$ be a continuous map of a compact metric space, and let $\{\grid_k\}$ be a collection of grids for $X$, partially ordered by refinement.  For each $k$, let $\fmgk{k}:\grid_k\rightrightarrows \grid_k$ be the minimal multivalued map associated to $f$ on $\grid_k$.  Then for every $\grid_{k_1} < \grid_{k_2}$, there are homomorphisms $(\hmg_{\grid_{k_1}}^{\grid_{k_2}})_*: H_*(\morsegk{k_1}) \to H_*(\morsegk{k_2})$ such that $(\hmg_{\grid_{k_2}}^{\grid_{k_3}})_* \circ (\hmg_{\grid_{k_1}}^{\grid_{k_2}})_* = (\hmg_{\grid_{k_1}}^{\grid_{k_3}})_*$ for any $\grid_{k_1} < \grid_{k_2} < \grid_{k_3}$.
\end{prop}

As we saw in Proposition~\ref{prop:MorseGraphChanges}, passing from a more refined grid $\grid$ to a coarser grid $\grid'$ can change $\morsegg{\grid}$ to $\morsegg{\grid'}$ by adding edges, adding  vertices, or merging  vertices.
We now examine the effect of each change on homology.  We first make some definitions for a Morse graph $\morseg$.  For a given vertex $v$, call the set of vertices above it $A(v):=\{u:u \succ v\}$, the vertices below $B(v)):=\{u:u \prec v\}$, and the  incomparable vertices $N(v):=\{u:u \incomp v\}$. For a pair of vertices $v \succ v'$, call the set of vertices in between $I(v,v'):=\{u: v\succ u \succ v'\} = B(v)\cap A(v')$.  Let $a(v)=\#(A(v))$ and $b(v)=\#(B(v))$.
%, and $i(v,v')=\#(I(v,v'))$.  Finally, let $\morseg(I(v,v'))$ be the subgraph of $\morseg$ consisting of the vertices in $I(v,v')$ and the edges between them.

Recall that when we merge two vertices $v$ and $v'$, we also merge the relevant edges (for example, edges $u\to v$ and $u \to v'$ are merged).

\begin{thm}
For grids $\grid < \grid'$, the graph $\morsegg{\grid'}$ can be obtained from $\morsegg{\grid}$ by a series of the following operations.
\begin{enumerate}
\item Adding a vertex (with no edges).
\item Adding an edge $v\to v'$, where $v$ and $v'$ are incomparable, $v$ is  maximal in $N(v')$ (that is, $A(v)\cap N(v')=\emptyset)$ and $v'$ is  minimal in $N(v)$ (that is, $B(v')\cap N(v)=\emptyset$).
\item Merging two vertices $v$ and $v'$, such that $v\succ v'$, there are no intermediate vertices (that is, $I(v,v')=\emptyset$), and $u\succ u'$ for all $u\in A(v')$ and $u' \in B(v)$.
\end{enumerate}
\end{thm}

\begin{proof}
If coarsening the grid adds a vertex with edges, we can clearly add the vertex first, then add the edges.  

If it adds an edge $u\to u'$, we claim that  $u$ and $u'$ must be incomparable.  If $u\succ u'$, then there is already an edge $u\to u'$, and we cannot add one.  If $u'\succ u$, then adding an edge $u\succ u'$ means that the corresponding Morse sets are now in the same equivalence class, and thus the vertices are merged.  So $u$ and $u'$ must be incomparable.  Since $\morsegg{\grid'}$ is a Morse graph, adding an edge from $u \to u'$  implies that the resulting graph will have edges from every vertex in $\{u\} \cup A(u)$ to every vertex in $\{u'\} \cup B(u')$.  We can add these edges iteratively, starting with an edge $v \to v'$ that does not force the addition of any other edges (that is, $v$ is  maximal in $N(v')$  and $v'$ is  minimal in $N(v)$), then continuing with the resulting graph.  We claim that we can find such an edge by picking $v$ to be a maximal element of $(\{u\}\cup A(u))\cap N(u')$ and $v'$ to be a minimal element of $(\{u'\}\cup B(u'))\cap N(v)$.  Assume, for the sake of contradiction, that there exists a $y \in A(v)\cap N(v')$. Then $y$ is not  in $N(u')$ (that would contradict maximality of $v$), so $y$ is comparable to $u'$ but not to $v'$.  Since $u'\succeq v'$, that means that $u' \succ y$, so we have $u' \succ y \succ v \succeq u$, meaning that $u$ and $u'$ are comparable, which is a contradiction.  Thus $A(v)\cap N(v') = \emptyset$.  And since $B(v')\subset \{u'\}\cup B(u')$, the minimality of $v'$ implies that $B(v')\cap N(v) = \emptyset$.

We can merge two incomparable vertices $v$ and $u$ by first adding an edge from $v$ to  $u$.  So we can assume that $v\succ u$.  If $I(v,u)\ne\emptyset$, then, since $\morsegg{\grid'}$ is a Morse graph, merging $v$ and $u$ into one vertex also merges everything in $I(v,u)$ with $v$.  We can perform these mergers iteratively, starting by verging $v$ with a maximal member $v'$ of $I(v,u)$, that is, one with $I(v,v')=\emptyset$, continuing with the resulting graph by merging the new vertex $v\sim v'$ with a maximal member of $I(v\sim v',u)$, and so on.

Merging $v$ and $v'$ forces $u\succ u'$ for all $u\in A(v')$ and $u' \in B(v)$ (we already have $w\succ w'$ for all $w\in A(v)$ and $w' \in B(v')$, since $v\succ v'$).  If any such edge $u \to u'$ does not already exist, that is, $u \not\succ u'$, then $u \incomp u'$: if $u'\succ u$, then we would have $v \succ u' \succ u \succ v'$, contradicting the assumption that there are no intermediate vertices between $v$ and $v'$.  So, by adding edges as in (2) before merging, we can assume that $u\succ u'$ for all $u\in A(v')$ and $u' \in B(v)$.

\end{proof}

So, to understand the effect on homology of passing from $\morsegg{\grid}$ to $\morsegg{\grid'}$ for grids $\grid<\grid'$, it is sufficient to understand the effects of the three operations above.

\begin{thm}
Let $\morseg$ and $\morseg'$ be Morse graphs.
\begin{enumerate}
\item If $\morseg'$ is obtained from $\morseg$ by adding a vertex (with no edges), then \newline $\dim H_0(\morseg') = \dim H_0(\morseg) +1 $ and \newline $\dim H_1(\morseg') = \dim H_1(\morseg)$.
\item  If $\morseg'$ is obtained from $\morseg$ by adding an edge $v\to v'$, where $v$ and $v'$ are incomparable, $v$ is  maximal in $N(v')$ (that is, $A(v)\cap N(v')=\emptyset)$ and $v'$ is  minimal in $N(v)$ (that is, $B(v')\cap N(v)=\emptyset$), then
	\begin{enumerate}
	\item if $v$ and $v'$ are in different connected components of $\morseg$, then \newline
		 $\dim H_0(\morseg') = \dim H_0(\morseg) -1 $ and \newline $\dim 			H_1(\morseg') = \dim H_1(\morseg)$.
		 
	\item if $v$ and $v'$ are in the same connected component of $\morseg$, then \newline $\dim H_0(\morseg') = \dim H_0(\morseg)$ and \newline $\dim H_1(\morseg') = \dim H_1(\morseg) +1$.
		
	\end{enumerate}
\item If $\morseg'$ is obtained from $\morseg$ by merging two vertices $v$ and $v'$, such that $v\succ v'$, there are no intermediate vertices (that is, $I(v,v')=\emptyset$), and $u\succ u'$ for all $u\in A(v')$ and $u' \in B(v)$, then \newline $\dim H_0(\morseg') = \dim H_0(\morseg)$ and \newline $\dim H_1(\morseg') = \dim H_1(\morseg)  - a(v) - b(v')$.

\end{enumerate}
\end{thm}

\begin{proof}
The result in (1) is immediate, since $H_0$ counts connected components and $H_1$ counts loops.

For (2), we observe that since $v$ is maximal in $N(v')$ and $v'$ is minimal in $N(v)$, adding the edge does not force the existence of any other edges.
Assume first that $v$ and $v'$ are in different connected components.  Then adding the edge $v\to v'$ joins the two components, reducing $\dim H_0$ by one, and contributes nothing to $\dim H_1$.

Now consider the case with $v$ and $v'$ in the same connected component.  Then adding the edge $e$ from $v$ to $v'$ does not change the number of components, so $\dim H_0$ does not change.  Since $v$ and $v'$ are in the same connected component, there is a path $p$ from $v'$ to $v$ in the undirected graph.  This gives a new cycle, $pe$.  If $p'$ is another path from $v'$ to $v$,  then in homology the loop $p'e= p'(-p)+pe$, so adding $e$ adds only one to $\dim H_1$. 

For (3), since $v\succ v'$, they are already in the same connected component, and merging them does not change $\dim H_0$.
We observe that because there are no intermediate vertices, merging $v$ and $v'$ does not merge any other vertices, and because $u\succ u'$ for all $u\in A(v')$ and $u' \in B(v)$, it does not create any new edges.  

We show that merging does not create any new loops.  Let $\bar v$ be the new merged vertex.  Any simple loop $\bar v u_1u_2\dots u_n \bar v$ in $\morseg'$ corresponds to a loop in $\morseg$ of the form $vu_1u_2\dots u_n v$, $v'u_1u_2\dots u_n v'$, $v'u_1u_2\dots u_n vv'$, or $vu_1u_2\dots u_n v'v$, since $v$ and $v'$ are adjacent in $\morseg $ (thought of as undirected).  

However, merging can eliminate loops.  If $vuv'v$ (or $v'uvv'$) is a nontrivial loop in $\morseg$, then it becomes the trivial $\bar v u \bar v$ in $\morseg'$.  If $u$ is above $v$ or below $v'$, then $vuv'v$ is such a loop, since $v\succ v'$. Conversely, for any nontrivial loop $vuv'v$, $u$ must be above $v$ or below $v'$; otherwise we would have $v\succ u\succ v'$, contradicting the assumption that there are no intermediate vertices between $v$ and $v'$.
The sets $A(v)$ and $B(v')$ are disjoint, since $v\succ v'$, so the total number of such loops is $a(v) + b(v')$.

We observe that if $vu_1\ldots u_nv$ and $v'u_1\ldots u_nv'$  are two loops in $\morseg$, then in homology $vu_1\ldots u_nv = v'u_1\ldots u_nv'  + v'u_nvv' + v'vu_1v'$.  Similarly, if $vu_1\ldots u_nv$ and $vu_1\ldots u_nv'v$ are two loops in $\morseg$, then in homology, $vu_1\ldots u_nv'v = vu_1\ldots u_nv + vu_nv'v$.  Thus collapsing a loop of the form $vuv'v$ to a trivial loop reduces $\dim H_1$ by only one.  Since there are $a(v)+b(v')$ such loops, the result follows.

\end{proof}

\section{Local persistence}\label{sect:localPers}

The Morse graph helps us to understand the gradient-like behavior of the system.  To understand the recurrent behavior, we study the dynamics on individual Morse sets.  (The Conley index can  rigorously relate these multivalued map dynamics to the dynamics of the underlying system; see, for example, \cites{DF,DFT,akkmop,bmmp,Conley,con1,con2,con3,con4,con5}.)

Let $\grid < \grid'$ be two grids, and $\morseag{\grid} = (\morsedg{\grid}, \{\fmgg{\grid}_j:\morsesg{\grid}_j \rightrightarrows \morsesg{\grid}_j\})$ and $\morseag{\grid'} = (\morsedg{\grid'}, \{\fmgg{\grid'}_j:\morsesg{\grid'}_j \rightrightarrows \morsesg{\grid'}_j\})$ be the corresponding augmented Morse decompositions.  We have seen that $f$ induces a morphism of finest augmented Morse decompositions $\ha:\morseag{\grid}\to \morseag{\grid'}$, that is, a pair $(\hm,\{h_j\})$, where $\hm$ is a morphism between the finest Morse decompositions $ (\{\morsesg{\grid}_j\},\succeq^{\grid})$ and $\morsedg{\grid'}= (\{\morsesg{\grid'}_j\},\succeq^{\grid'})$, and for each Morse set $\morsesg{\grid}_j$, $h_j: \morsesg{\grid}_j\to\hm(\morsesg{\grid}_j)$ is a grid map morphism.  We will consider an individual Morse set $\morsesg{\grid}_j$, which we will generally denote simply by $\morses$; we denote the corresponding Morse set in $\grid'$, $\hm(\morses)$, by $\morses'$,  the map $h_j:\morses\to\morses'$ simply by $h$, and the maps $\fmgg{\grid}_j:\morses \rightrightarrows \morses$ and $\fmgg{\grid'}_j:\morses' \rightrightarrows \morses'$ by $\fm$ and $\fm'$, respectively.

As we pass from  $\grid$ to the coarser  $\grid'$, a Morse set can change in two ways.  The geometric realization of the Morse set can  grow, that is, $|\morses| \subsetneq |\morses'|$, and the dynamics can change as new edges  appear, that is, $h(\fm(G)) \subsetneq \fm'(h(G))$ for some $G\in\morses$.  We first discuss the case where the Morse set grows.

Just the mapping $h$ which takes grid elements in $\morses\subset\grid$ to coarser elements in $\morses'\subset\grid'$ can cause the geometric realization of the Morse set to grow.  We will have $|\morses| \subsetneq |h(\morses)|$ if there are $\grid$-grid elements $G\in\morses$ and $H\not\in\morses$  such that $|G|$ and $|H|$ are both subsets of $|h(G)|$; then $|\morses|$ does not include all of $|h(G)|$.  For example, in Figure~\ref{fig:gridMapEx}(A), we have $\morses = \{G_2\}$, but $|h(\morses)| = G_1\cup G_2$ in (B).

\begin{prop}
$h(\morses)$ is an $\fm'$-invariant set.
\end{prop}

\begin{proof}
This follows from the fact that $h$ is a grid map morphism, so $h(\fm(G)) \subset \fm'(h(G))$ for all $G\in\morses$.
We need to show that $h(\morses)\subset \fm'(h(\morses))$ and $h(\morses)\subset (\fm')^{-1}(h(\morses))$.  For the first, we have $h(\morses) = h(\fm(\morses)) \subset \fm'(h(\morses))$.  For the second, the fact that $h$ is a grid map morphism implies that $h(\fm^{-1}(\morses)) \subset (\fm')^{-1}(h(\morses))$, so we have $h(\morses) = h(\fm^{-1}(\morses)) \subset (\fm')^{-1}(h(\morses))$.
\end{proof}

When the mapping $h:\morses\to\morses'$ is not surjective, $h(\morses)\subsetneq\morses'$, we see a change in the recurrent dynamics as we pass from $\grid$ to $\grid'$.
To better understand these changes in the size and shape of the Morse set, we can look at the persistent homology.  In \cite{djkklm}, the authors define persistent homology for Morse sets as finite topological spaces, in part to address the issue that in the grid setting, different Morse sets can have nonempty intersection.  Since we are studying Morse sets individually, we can also use cubical homology \cites{chomp,MrozekCubHomSurvey}.

Even when passing from $\fm$ to $\fm'$ does not increase the size of the Morse set, that is, even when $h(\morses)=\morses'$, it can still change the dynamics by growing the image, $h(\fm(G)) \subsetneq \fm'(h(G))$ for some $G\in\morses$.  Equivalently, there are edges in $\fm':\morses'\to\morses'$ that do not come from edges in $\fm:\morses\to\morses$.  (This will not happen for the minimal multivalued map associated to $f:X\to X$, since there is an $\fm'$-edge $G'\to H'$ if and only if there is an $\fm$-edge $G\to H$ for some $G\in h^{-1}(G')$ and $H \in h^{-1}(H')$, but it can occur for other methods of defining multivalued grip maps, discussed in Section~\ref{sect:timeSeries}.)

So, for the rest of this section, we assume that $h:\morses\to\morses'$ is surjective.  We cannot compare $\fm:\morses'\to\morses'$ and $\fm:\morses'\to\morses'$ directly, since they act on different spaces.  To understand the change in dynamics, we compare $\fm':\morses'\to\morses'$ to $\bar\fm:\morses\to\morses$, the map induced by $\fm$ on $\morses'$.  

More precisely, we define $\bar\fm := h \circ \fm \circ h^{-1}$.  Thus, thinking of $\bar\fm$ as a graph, we have an edge $G'\to H'$ if and only if $\fm$ has an edge $G\to H$ for some $G$ and $H$ such that $h(G)=G'$ and $h(H)=H'$.  We observe that $\bar\fm \subset \fm'$, that is, $\bar\fm(G') \subset \fm'(G')$ for all $G'\in\morses'$.

Since the Morse sets correspond to the strongly connected components of the multivalued map considered as a graph, the dynamics on any Morse set are transitive, meaning that for any $G$, $H \in \morses$, there is an $n>0$ such that $H\subset\fm^n(G)$.  Thus passing from $\fm$ to $\fm'$ will not cause a radical change in the dynamics; they will still be transitive.  However, it can cause one important change.  Recall that $\fm$ is mixing if there exists an $n>0$ such that for any $G$, $H\in\morses$ and any $m\ge n$, $H\subset\fm^m(G)$ (\cite{lm}*{\S4.5}).  Any mixing map is transitive, but not vice versa.  If $\fm$ is transitive but not mixing, then we can partition $\morses$ into subsets $\morses_1,\ldots,\morses_d$ for some $d$, where $\fm$ cyclically permutes the sets $\morses_i$ ($\fm(\morses_i) = \morses_{i+1\mod d}$)
and the restriction of the $d$th power of $\fm$ to each $\morses_i$, $\fm^d|_{\morses_i}:\morses_i\to\morses_i$, is mixing.  If $d$ is the minimal such integer, we call $d$ the period of $\fm$, $\per(\fm)$, and say that $\fm$ is $d$-periodic (so a mixing map is 1-periodic).  We see in Figure~\ref{fig:Mixing} that passing from $\fm$ to $\fm'$ can change the dynamics from non-mixing to mixing.

\begin{figure}

\begin{subfigure}{.5\textwidth}
\centering
\begin{tikzpicture}%[xscale=1]

\node[draw,circle] (v12) at (0,0) {$v_{12}$};
\node[draw,circle] (v11) at (0,2) {$v_{11}$};
\node[draw,circle,fill=lightgray] (v22) at (4,0) {$v_{22}$};
\node[draw,circle,fill=lightgray] (v21) at (4,2) {$v_{21}$};

\draw (v12) edge[-Latex, thick, bend right]  (v22);
\draw (v11) edge[-Latex, thick, bend right]  (v21);
\draw (v11) edge[-Latex, thick, bend right]  (v22);

\draw (v21) edge[-Latex, thick, bend right, color=red]  (v11);
\draw (v21) edge[-Latex, thick, bend right, color=red]  (v12);
\draw (v22) edge[-Latex, thick, bend right, color=red]  (v11);

\end{tikzpicture}
\caption{$\fm$ is 2-periodic.}
\label{fig:exMix1}
\end{subfigure}

\begin{subfigure}{.5\textwidth}
\centering
\begin{tikzpicture}%[xscale=5]

\node[draw,circle] (v12) at (0,0) {$v_{12}$};
\node[draw,circle] (v11) at (0,2) {$v_{11}$};
\node[draw,circle,fill=lightgray] (v22) at (4,0) {$v_{22}$};
\node[draw,circle,fill=lightgray] (v21) at (4,2) {$v_{21}$};

\draw (v12) edge[-Latex, thick, bend right]  (v22);
\draw (v11) edge[-Latex, thick, bend right]  (v21);
\draw (v11) edge[-Latex, thick, bend right]  (v22);

\draw (v21) edge[-Latex, thick, bend right, color=red]  (v11);
\draw (v21) edge[-Latex, thick, bend right, color=red]  (v12);
\draw (v22) edge[-Latex, thick, bend right, color=red]  (v11);

\draw (v11) edge[-Latex, thick, bend right, color=blue]  (v12);

\end{tikzpicture}
\caption{$\fm'$ is mixing.}
\label{fig:exMix2}
\end{subfigure}

\caption{Passing from $\grid$ to $\grid'$ can make the dynamics mixing.}
\label{fig:Mixing}
\end{figure}
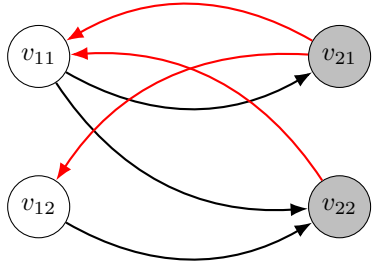
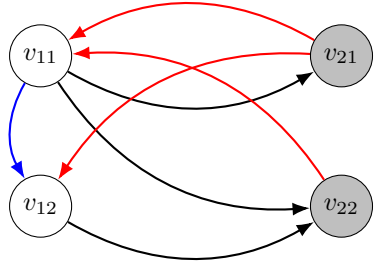

\begin{prop}\label{prop:periods}
$\per(\bar\fm)$ divides $\per(\fm)$, and $\per(\fm')$ divides $\per(\bar\fm)$.
\end{prop}

\begin{proof}
We use the fact that $\per(\fm)$ is the greatest common divisor of the set of lengths of closed $\fm$-orbits (\cite{lm}*{\S4.5}).   Since every $\fm$-orbit maps to an $\bar\fm$-orbit, and every $\bar\fm$-orbit is an $\fm'$-orbit, the result follows.
\end{proof}

When $h$ is surjective, we see that passing from $\grid$ to $\grid'$ can transform a non-mixing $\fm$ into a mixing $\fm'$ in two ways.  Either $h$ can combine two vertices that are in different elements of the periodic partition for $\fm$, or $\fm'$ can add an edge between two such vertices.  More precisely, we have the following.

\begin{thm}
Let $h:\morses\to\morses'$ be surjective.
\begin{enumerate}
\item Let $\fm$ be $d$-periodic ($d>1$) and  let $h(G)=h(H)$ for some $G\in\morses_i$ and $H\in\morses_j$, where $\morses_i$ and $\morses_j$ are elements of the $d$-periodic partition of $\morses$ and $i-j$ is relatively prime to $d$.  Then $\bar\fm$ is mixing.

\item Let $\bar\fm$ be $d$-periodic ($d>1$) and  let $\fm'$ add an edge from $G'$ to $H'$ for some $G'\in\morses'_i$ and $H'\in\morses'_j$, where $\morses'_i$ and $\morses'_j$ are elements of the $d$-periodic partition of $\morses'$ and $i-j +1$ is relatively prime to $d$.  (That is, $H'\in\fm'(G')$ but $H'\not\in\bar\fm(G')$.)  Then $\fm'$ is mixing.

\end{enumerate}
\end{thm}

Observe that by Proposition~\ref{prop:periods}, $\fm'$ is mixing if $\bar\fm$ is.  Figure~\ref{fig:Mixing} shows an example of part (2) of the theorem.

\begin{proof}
We again use the fact that the period is the greatest common denominator of the set of lengths of closed orbits.  For notational convenience, assume that we have ordered the periodic partition elements so that $i>j$.

For (1), observe that $\fm^{i-j}(\morses_j)=\morses_i$, so, in particular, there is a $G_0\in\morses_i$ such that $G_0 \in \fm^{i-j}(H)$.  Since $\fm^d|_{\morses_i}:\morses_i\to\morses_i$ is mixing, for some $n>0$, there is an $\fm$-orbit  of length $nd$ from $G_0$ to $G$.  By concatenating, we get an $\fm$-orbit of length $i-j+nd$ from $H$ to $G$; its image under $h$ is a closed $\bar\fm$-orbit from $h(H)\sim h(G)$ to itself of the same length.  Since $\gcd(d, i-j+nd) = 1$, $\bar\fm$ is mixing.

The proof of (2) is similar.  There is a $G'_0 \in \morses'_i$ such that $G'_0 \in \bar\fm^{i-j-1}(H)$.  Since $\bar\fm^d|_{\morses'_i}:\morses'_i\to\morses'_i$ is mixing, for some $n>0$, there is an $\bar\fm$-orbit  of length $nd$ from $G'_0$ to $G'$.  Concatenating the orbit from $H'$ to $G'_0$, the orbit from $G'_0$ to $G'$, and the (length-one) $\fm'$-orbit from $G'$ to $H'$,
and recalling that every $\bar\fm$-orbit is \emph{a fortiori} an $\fm'$-orbit, we get a closed orbit of length $i-j+nd+1$.  Again, since  $\gcd(d, i-j+1+nd) = 1$, $\fm'$ is mixing.

\end{proof}

%define adjacency matrix.
%
%\begin{prop}
%Eigenvalues, non to mixing.
%\end{prop}
%
%h onto
%
%dynamics
%
%compare to image dynamics.  biggest question - mixing or not?  Can look at eigenvalues.

%\section{Comparing Morse decompositions for $f$ and $\fm$}
%
%Skip?
%
%Misch, Akin, Osipenko
%
%morphism from CR to grid.  In general no morphism from Morse decomp for f to Morse decomp for F - connecting orbit points might not go to recurrent points.  There exists a grid that works.
%
%"Persistence" from actual Morse decomp to grid decomp?  Limits on CR, basic sets either chain mixing, periodic, or adding machine - persistence.

\section{Time series} \label{sect:timeSeries}

In applications, we may not have complete information about the map $f:X\to X$, or it may be too difficult to compute.
We can also get multivalued grid maps from time series, either from sampling the phase space directly or from observations.  Reconstructing dynamics from time series is a very active area of research.  See, for example, \cites{ABMS,GBM,bmmp,MMRS,PEREIRA20156026,MZR,PerNotices,Gholizadeh2018ASS,baker}.  In particular, in \cite{mik} the authors create a Morse decomposition for the time series map.

\subsection{Observations}

We largely follow the presentation in \cite{mik}.
Let $\pi:X\to\R^m$ be a continuous map.  We think of $\pi$ as an observation of the state of the system.  It can measure $m$ different scalar quantities, or a single quantity $\pi_0$ at $m$ time steps:  $$\pi(x) = (\pi_0(x), \pi_0(f(x)),\ldots,\pi_0(f^{m-1}(x))).$$

Our data consist of observations of a finite number of finite orbit segments.  More precisely, let $\{x_i\}_{i=1}^K$ be a set of initial points in $X$.  For each $x_i$, we take observations from the first $N_i$ points on the orbit of $x_i$, and define $y_i^j := \pi(f^j(x_i))$ for $0 \le j \le N_i-1$.  Let ${\mathbf y}_i = (y_i^0,\ldots,y_i^{N-i-1})$ denote the observations from initial point $x_i$, and $O=\{{\mathbf y}_i\}$ the set of all observations.

Now take a grid $\grid$ on $\pi(X)\subset\R^m$.  We define the multivalued map $\fmg_O:\grid\rightrightarrows\grid$ associated to $O$ on $\grid$ by
 \begin{align*}
 \fmg_O(G) = \{ H\in\grid: \ & \text{there exist ${\mathbf y}_i\in O$ and $j$, $0\le j<N_i-1$,}\\ &\text{such that   $y_i^j\in G$, $y_i^{j+1}\in H$} \}.
 \end{align*}
As with the minimal multivalued map associated to $f$ in Section~\ref{sect:grids}, we trim any stranded vertices from the graph associated to $\fmg_O$.

We can define a Morse decomposition for $\fmg_O$ just as before:  Let $\morsed_O$ be the finest Morse decomposition associated to $O$, corresponding to the strongly connected components of $\fmg_O$.
We now show that we get persistence over grid refinement.  We begin with the following analogue of Proposition~\ref{prop:gridmappersist}.

\begin{prop}\label{prop:gridmappersistO}
Let $O$ be a set of observations, and let $\{\grid_k\}$ be a collection of grids for $\pi(X)$, partially ordered by refinement.  For each $k$, let $\fmgk{k}_O:\grid_k\rightrightarrows \grid_k$ be the multivalued map associated to $O$ on  $\grid_k$.
Then for every $\grid_{k_1} < \grid_{k_2}$, there is a morphism $h_{k1}^{k2}:\grid_{k_1} \to \grid_{k_2}$, and if $\grid_{k_1} < \grid_{k_2} < \grid_{k_3}$, then $h_{k_2}^{k_3}\circ h_{k_1}^{k_2} = h_{k_1}^{k_3}$.

\end{prop}

\begin{proof}
The proof is similar to that of Proposition~\ref{prop:gridmappersist}.  Take grids $\grid < \grid'$.  Since $\grid$ is a refinement of $\grid'$, by definition for any $G\in\grid$ there is a $G'\in\grid'$ such that $G\subset G'$; define $h$ by $h(G) := G'$.  If $y_i^j\in G$ and $y_i^{j+1}\in H$, then $y_i^j\in h(G)$ and $y_i^{j+1}\in h(H)$.  Thus we have $h(\fm_O(G)) \subset \fm'_O(h(G))$.  It follows from the definition that  $h_{k_2}^{k_3}\circ h_{k_1}^{k_2} = h_{k_1}^{k_3}$.
\end{proof}

\begin{cor}
Let $O$ be a set of observations, and let $\{\grid_k\}$ be a collection of grids for $\pi(X)$, partially ordered by refinement.  For each $k$, let $\fmgk{k}_O:\grid_k\rightrightarrows \grid_k$ be the multivalued map associated to $O$ on  $\grid_k$.  Then for every $\grid_{k_1} < \grid_{k_2}$, there is a morphism $\ha_{\grid_{k_1}}^{\grid_{k_2}}: \morseak{k_1} \to \morseak{k_2}$ such that $\ha_{\grid_{k_2}}^{\grid_{k_3}} \circ \ha_{\grid_{k_1}}^{\grid_{k_2}} = \ha_{\grid_{k_1}}^{\grid_{k_3}}$ for any $\grid_{k_1} < \grid_{k_2} < \grid_{k_3}$.
\end{cor}

\begin{proof}
This follows from Proposition~\ref{prop:gridmappersistO} and Lemma~\ref{lem:augMorsePersist}.
\end{proof}

As a consequence, all of the results on global and local persistence from Sections~\ref{sect:globalPers} and \ref{sect:localPers} apply to the maps $\fmg_O$ arising from time series of observations.

\subsection{Sampled dynamics}

Instead of taking observations, we can also gather information about the map $f:X\to X$ by sampling the phase space directly.
Abstractly, in our setting, the two approaches are essentially the same, but they are conceptually different ways of looking at the dynamics, so  we treat them separately.  We generally follow the presentation in \cite{MMRS}.

Our data consist of a sample of the dynamics, in the form of pairs of points, $D:=\{(x_i,y_i): y_i=f(x_i) \}_{i=1}^K$.  For a grid $\grid$ on $X$, we define the multivalued map $\fmg_D:\grid\rightrightarrows\grid$ associated to $D$ on $\grid$ by
$$\fmg_D(G) = \{H \in \grid: \text{ there exists $(x_i,y_i)\in D$ with $x_i\in G$, $y_i\in H$} \}.$$

We proceed just as we did with the observations. We trim any stranded vertices from the graph associated to $\fmg_D$.
We define $\morsed_D$ to be the finest Morse decomposition associated to $D$, corresponding to the strongly connected components of $\fmg_D$.
And we get persistence over grid refinement. 

\begin{prop}\label{prop:gridmappersistD}
Let $D$ be a set of sampled data, and let $\{\grid_k\}$ be a collection of grids for $X$, partially ordered by refinement.  For each $k$, let $\fmgk{k}_D:\grid_k\rightrightarrows \grid_k$ be the multivalued map associated to $D$ on  $\grid_k$.
Then for every $\grid_{k_1} < \grid_{k_2}$, there is a morphism $h_{k1}^{k2}:\grid_{k_1} \to \grid_{k_2}$, and if $\grid_{k_1} < \grid_{k_2} < \grid_{k_3}$, then $h_{k_2}^{k_3}\circ h_{k_1}^{k_2} = h_{k_1}^{k_3}$.

\end{prop}

\begin{proof}
The proof is essentially identical to that of Proposition~\ref{prop:gridmappersistO}.  Take grids $\grid < \grid'$.  Since $\grid$ is a refinement of $\grid'$, by definition for any $G\in\grid$ there is a $G'\in\grid'$ such that $G\subset G'$; define $h$ by $h(G) := G'$.  If $x_i\in G$ and $y_i\in H$, then $x_i\in h(G)$ and $y_i\in h(H)$.  Thus we have $h(\fm_D(G)) \subset \fm'_D(h(G))$.  It follows from the definition that  $h_{k_2}^{k_3}\circ h_{k_1}^{k_2} = h_{k_1}^{k_3}$.
\end{proof}

\begin{cor}
Let $D$ be a set of sampled data, and let $\{\grid_k\}$ be a collection of grids for $X$, partially ordered by refinement.  For each $k$, let $\fmgk{k}_D:\grid_k\rightrightarrows \grid_k$ be the multivalued map associated to $D$ on  $\grid_k$.  Then for every $\grid_{k_1} < \grid_{k_2}$, there is a morphism $\ha_{\grid_{k_1}}^{\grid_{k_2}}: \morseak{k_1} \to \morseak{k_2}$ such that $\ha_{\grid_{k_2}}^{\grid_{k_3}} \circ \ha_{\grid_{k_1}}^{\grid_{k_2}} = \ha_{\grid_{k_1}}^{\grid_{k_3}}$ for any $\grid_{k_1} < \grid_{k_2} < \grid_{k_3}$.
\end{cor}

\begin{proof}
This follows from Proposition~\ref{prop:gridmappersistD} and Lemma~\ref{lem:augMorsePersist}.
\end{proof}

Thus all of the results on persistence from Sections~\ref{sect:globalPers} and \ref{sect:localPers} apply to the maps $\fmg_D$ arising from time series of sampled dynamics as well.

\subsection{Generalizations}

In applications, the time series data may be noisy, and we may not want to construct grid maps with images corresponding to every data point, as we do above.  Instead, we may require that they meet some type of frequency threshold.  We discuss two such methods, and their relation to persistence.

One such method is discussed in \cite{djkklm}.  Translating their approach slightly to our setting, of a grid $\grid$ and sampled data points $D=\{(x_i,y_i): y_i=f(x_i) \}_{i=1}^K$, they define functions $n(G,H) := \#\{x_i: x_i\in G, y_i\in H\}$ and $n_{max}(\grid):= \max_{G,H\in\grid}\{n(G,H)\}$.  They then define the grid map ${\hatfmg_D}:\grid\rightrightarrows\grid$  by
$${\hatfmg_D}(G):= \{H:n(G,H) > \mu\cdot n_{max}(\grid) \}, $$
for some threshold parameter $\mu$.  (They then take the convex hull, but we can ignore that for our purposes.)

For a fixed threshold $\mu$, this method will not give persistence over grid refinement.  If $\grid<\grid'$ and $H\in{\hatfmg_D}(G)$, it will not necessarily be 
true that $h(H) \in {\hatfmgg{\grid'}_D}(h(G))$.  Let $M=M(\grid,\grid') := \max_{G'\in\grid'}\#(h^{-1}(G'))$, the maximum number of $\grid$-grid elements any $\grid'$-grid element is subdivided into.
Then $n_{max}(\grid')$ can be as great as $M^2\cdot n_{max}(\grid)$, if there are elements $G'$, $H'\in\grid'$ such that $n(G,H) = n_{max}(\grid)$ for all $G \in h^{-1}(G')$ and $H \in h^{-1}(H')$.  At the same time, there can be $G'$, $H'\in\grid'$ such that $n(G',H') = n(G,H)$ for some $G \in h^{-1}(G')$ and $H \in h^{-1}(H')$, if there are no data points going from any other element of $h^{-1}(G')$ to any other element of $h^{-1}(H')$.  Thus $n(G,H) > \mu\cdot n_{max}(\grid)$ does not imply $n(h(G),h(H)) > \mu\cdot n_{max}(\grid')$.

However, if we can choose the threshold $\mu$ to depend on the grids $\grid$ and $\grid'$ (and \emph{not} on the data $D$), we can get persistence.  Since $n(h(G),h(H)) \ge n(G,H)$ and $n_{max}(\grid')\le M^2\cdot n_{max}(\grid)$, if we choose a threshold $\mu(\grid') \le \frac{\mu(\grid)}{M^2}$, we will get persistence.

Morita et al.\ present another method in \cite{mik}.  They define $v(G) := \#\{x_i:x_i \in G\}$ and  the transition probability $t(G,H): =\frac{n(G,H)}{v(G)}$. They then define the grid map ${\tildefmg_D}:\grid\rightrightarrows\grid$  by
$${\tildefmg_D}(G):= \{H:t(G,H) \ge \lambda\cdot t(H,G) \text{ and } n(G,H)\ge \mu \}, $$
for some threshold parameters $\lambda$ and $\mu$.  This definition will not give persistence, and there is no way to get persistence by choosing thresholds based solely on the grids $\grid$ and $\grid'$: we can have $t(G,H) > t(H,G)$, but $\frac{t(h(G),h(H))}{t(h(H),h(G))}$ arbitrarily small, if there are many data points going from the other elements of $h^{-1}(H)$ to the other elements of $h^{-1}(G)$.  However, 
 we may be able to get persistence if we construct thresholds based on characteristics of the data as well as of the grids.  This is not necessarily an unreasonable requirement, given that in practice we may want to choose refinements based on the density of the data points anyway.

Both of these results can be generalized.  Characterizing persistent methods for constructing grid maps from time series will be the subject of future work.

%\nocite{*}

%\bibliography{morse}


\end{document}